\documentclass[11pt,letterpaper]{amsart}
\usepackage{amsmath}
\usepackage{amstext}
\usepackage{amssymb}
\usepackage{amsfonts}
\usepackage{enumerate}
\usepackage{mathrsfs}
\usepackage{braket}
\usepackage{xcolor}
\usepackage[all]{xy}
\usepackage{url}
\usepackage[T1]{fontenc}
\usepackage{stmaryrd}

\numberwithin{equation}{section}

\newtheoremstyle{note}
{1em}
{1em}
{}
{}
{\bfseries}
{:}
{.5em}
{}

\newtheorem{theorem}{Theorem}[section]
\newtheorem{lemma}[theorem]{Lemma}
\newtheorem{proposition}[theorem]{Proposition}
\newtheorem{corollary}[theorem]{Corollary}

\theoremstyle{note}
\newtheorem{remark}[theorem]{Remark}

\newtheorem{definition}[theorem]{Definition}

\newcommand{\N}{{\mathbb{N}}}

\newcommand{\n}[1]{ \left\|#1\right\| }

\newcommand{\pair}[2]{{\langle #1, #2 \rangle}}

\DeclareMathOperator{\HS}{HS}
\DeclareMathOperator{\tr}{tr}

\DeclareMathOperator{\spa}{span}
\DeclareMathOperator{\rank}{rank}

\DeclareMathOperator{\asdim}{asdim}
\DeclareMathOperator{\diam}{diam}
\DeclareMathOperator{\dist}{dist}
\DeclareMathOperator{\id}{Id}

\title[Quantum Asymptotic Dimension]{Asymptotic dimension and coarse embeddings in the quantum setting}

\author{Javier Alejandro Ch\'avez-Dom\'inguez}
\address{Department of Mathematics, University of Oklahoma, Norman, OK 73019-3103, USA} 
\email{jachavezd@ou.edu}
\thanks{The first-named author was partially supported by NSF grant DMS-1900985.}

\author{Andrew T. Swift}
\address{Department of Mathematics, University of Oklahoma, Norman, OK 73019-3103,
USA}
\email{ats0@ou.edu}


\begin{document}

\begin{abstract}
We generalize the notions of asymptotic dimension and coarse embeddings from metric spaces to quantum metric spaces in the sense of Kuperberg and Weaver \cite{KuperbergWeaver}.
We show that  quantum asymptotic dimension behaves well with respect to metric quotients and direct sums, and is preserved under quantum coarse embeddings.
Moreover, we prove that a quantum metric space that equi-coarsely contains a sequence of expanders must have infinite asymptotic dimension.
This is done by proving a quantum version of a vertex-isoperimetric inequality for expanders, based upon a previously known edge-isoperimetric one from \cite{TKRWV}.
\end{abstract}

\maketitle
\section{Introduction}
In \cite{KuperbergWeaver} the authors explore a generalization of metric spaces, called quantum metric spaces, which is related to the quantum graphs of quantum information theory; and they construct many generalizations of familiar metric space concepts, including a generalization of Lipschtiz map they call a co-Lipschitz morphism.  The purpose of this paper is to contribute to the structure theory of quantum metric spaces by applying ideas coming from the large-scale or coarse geometry of classical metric spaces.  Specifically, we propose generalizations of coarse embeddings and asymptotic dimension, and then prove some fundamental results about them.  Coarse geometry is an important area of mathematics with applications in group theory, Banach space theory, and computer science.  It began as a field of study with the polynomial growth theorem of Gromov \cite{gromov1981}, and the definition of coarse embedding and asymptotic dimension are also both due to Gromov \cite{Gromov1993}, although large-scale geometric ideas appear as early as the late 1960's in the original proof of Mostow's rigidity theorem \cite{Mostow1968}.  We refer to \cite{Nowak-Yu} for an excellent introduction to the subject.

Some notions from coarse geometry have already been explored in the noncommutative setting, see e.g. \cite{banerjee2015coarse,banerjee2016noncommutative}.
There are two important differences between these works and the present paper:  One is that they define and study noncommutative versions of coarse spaces whereas we are studying noncommutative metric spaces in a coarse fashion, but more importantly their approach is $C^*$-algebraic while ours (as clearly stated by the title of \cite{KuperbergWeaver}) is a von Neumann algebra one.
A number of noncommutative notions of topological dimension for $C^*$-algebras have been also studied in the literature, see e.g. \cite{Rieffel,Brown-Pedersen,Kirchberg-Winter,Winter,Winter-Zacharias}.
Let us emphasize in particular the nuclear dimension from \cite{Winter-Zacharias}, 
which is linked to coarse geometry:  For a discrete metric space of bounded geometry, the nuclear dimension of the associated uniform Roe algebra is dominated by the asymptotic dimension of the underlying space.
Once again, the approach to dimension in the present work is significantly different from the aforementioned ones because we are following the von Neumann algebra path.

In Section \ref{sec:definitions} we recall the definitions we need from \cite{KuperbergWeaver}, including quantum metric, distance, and diameter.  In Section \ref{sec:moduli} we generalize the definition of coarse embedding to quantum metric spaces using alternative versions of the usual moduli of expansion (or uniform continuity) and compression defined for classical functions.  It is shown that the moduli for classical functions and their canonically induced quantum functions coincide.  In Section \ref{sec:asdim} we generalize the definition of asymptotic dimension to quantum metric spaces and show that asymptotic dimension is preserved under coarse embeddings.  We show as a consequence that the asymptotic dimension of a quotient of a quantum metric space is no greater than the asymptotic dimension of the original space and that the asymptotic dimension of a direct sum of quantum metric spaces is equal to the maximum of the asymptotic dimensions of its summands.  We finish Section \ref{sec:asdim} by establishing the corresponding inequality for arbitrary sums of quantum metric spaces in the case when the sum is a reflexive quantum metric space.  In Section \ref{sec:expanders}, we show that quantum expanders satisfy a quantum version of a vertex isoperimetric inequality.  This can be used to show that a quantum metric space has infinite asymptotic dimension if it equi-coarsely contains a sequence of reflexive quantum expanders.  In particular, this includes the case when the  sequence of quantum expanders is induced by a sequence of classical expanders.

\section{Definitions}\label{sec:definitions}
We use the definitions of quantum metric space and related notions found in \cite{KuperbergWeaver}.  Just as metrics are defined on sets, quantum metrics are defined on von Neumann algebras, and classical metrics on a set $X$ are in a natural one-to-one correspondence with quantum metrics on $\ell_\infty(X)$.  We view von Neumann algebras as subsets of some space $\mathcal{B}(\mathcal{H})$ of bounded linear operators  on a Hilbert space $\mathcal{H}$. Given the von Neumann algebras $\mathcal{M}\subseteq \mathcal{B}(\mathcal{H}_1)$, $\mathcal{N}\subseteq \mathcal{B}(\mathcal{H}_2)$, we denote by $\mathcal{M}\overline{\otimes}\mathcal{N}$ their normal spatial tensor product, that is, the weak*-closure of $\mathcal{M}\otimes\mathcal{N}$ in $\mathcal{B}(\mathcal{H}_1\otimes_2 \mathcal{H}_2)$.  Given a von Neumann algebra $\mathcal{M}$, we denote the commutant of $\mathcal{M}$ by $\mathcal{M}'$.  An orthogonal projection in a von Neumann algebra will simply be called a projection.

\begin{definition}[{\cite[Definition 2.3]{KuperbergWeaver}}]
A \emph{quantum metric} on a von Neumann algebra $\mathcal{M}\subseteq \mathcal{B}(\mathcal{H})$ is a family $\mathbb{V}=\{\mathcal{V}_t\}_{t\in [0,\infty)}$ of weak*-closed subspaces of $\mathcal{B}(\mathcal{H})$ such that $\mathcal{V}_0=\mathcal{M}'$ and for all $t\in [0,\infty)$,
\begin{itemize}
\item $\mathcal{V}_t$ is self-adjoint.
\item $\mathcal{V}_s\mathcal{V}_t\subseteq\mathcal{V}_{s+t}$ for all $s\in [0,\infty)$.
\item $\mathcal{V}_t=\bigcap_{s>t}\mathcal{V}_s$.
\end{itemize}
A \emph{quantum metric space} is a pair $(\mathcal{M}, \mathbb{V})$ of a von Neumann algebra $\mathcal{M}$ with a quantum metric $\mathbb{V}$ defined on it.  We will often simply call a quantum metric space by its von Neumann algebra if there is no ambiguity regarding the quantum metric being considered.
\end{definition}

Given a metric space $(X,d)$, the canonical quantum metric space associated to it \cite[Proposition 2.5]{KuperbergWeaver} is $(\ell_\infty(X), \{\mathcal{V}_t\}_{t\in [0,\infty)})$, where \[\mathcal{V}_t=\big\{A\in \mathcal{B}(\ell_2(X))\mid \langle Ae_x,e_y \rangle=0\mbox{ for all } (x,y)\notin d^{-1}[0,t]\big\}\] for all $t\in[0,\infty ).$  Here $(e_x)_{x\in X}$ is the canonical basis of $\ell_2(X)$ and $\ell_\infty (X)$ is viewed as a subset of $\mathcal{B}(\ell_2(X))$ as diagonal operators in the standard way, that is, via the map $E\colon \ell_\infty\to \mathcal{B}(\ell_2(X))$ defined by $E(\phi)[f](x)=\phi(x)f(x)$ for all $\phi\in \ell_\infty(X)$, all $f\in\ell_2(X)$, and all $x\in X$.  It is in this way that quantum metric spaces are generalizations of classical metric spaces.
A natural distance function can be defined for projections in a von Neumann algebra that generalizes the notion of distance between subsets of a classical metric space.

\begin{definition}[{\cite[Definition 2.6]{KuperbergWeaver}}]
\label{def:distance}
Given a quantum metric $\mathbb{V}=\{\mathcal{V}_t\}_{t\in[0,\infty
)}$ on a von Neumann algebra $\mathcal{M}$, the \emph{distance} between two projections $P$ and $Q$ in $\mathcal{M}\overline{\otimes}\mathcal{B}(\ell_2)$ is 
\[\dist_\mathbb{V}(P,Q)=\inf\{t\mid P(A\otimes \id)Q\neq 0 \mbox{ for some }A\in \mathcal{V}_t\}.\]
Identifying $\mathcal{M}$ with $\mathcal{M}\otimes \id\subseteq \mathcal{M}\overline{\otimes} \mathcal{B}(\ell_2)$ yields the equivalent formula
\[\dist_\mathbb{V}(P,Q)=\inf\{t\mid PAQ\neq 0 \mbox{ for some }A\in \mathcal{V}_t\}\]
for projections $P,Q$ in $\mathcal{M}$.  The subscript will usually be omitted if there is no ambiguity regarding the quantum metric being used.
\end{definition}

We refer to \cite[Definition 2.7]{KuperbergWeaver} and \cite[Proposition 2.8]{KuperbergWeaver} for basic properties of quantum distance.  In particular, the distance function associated to a quantum metric satisfies an analog of the triangle inequality.  If $\mathcal{M}$ is a quantum metric space and $P,Q, R$ are projections in $\mathcal{M}\overline{\otimes}\mathcal{B}(\ell_2)$, then $\dist (P,Q)\leq \dist (P,R)+\sup \{\dist (\tilde{R},Q) \mid R\tilde{R}\neq 0\}$, where $\tilde{R}$ ranges over projections in $\mathcal{M}\overline{\otimes}\mathcal{B}(\ell_2)$.  Note that the same proof found in \cite[Proposition 2.8]{KuperbergWeaver} shows that if $P,Q,R$ are projections in $\mathcal{M}$, then $\tilde{R}$ may be taken to range only over projections in $\mathcal{M}$.   

It is not hard to see that if $(\ell_\infty(X),\mathbb{V})$ is the canonical quantum metric space associated to metric space $(X,d)$, then $\dist_\mathbb{V}(\chi_S,\chi_T)=d(S,T)$ for any subsets $S,T$ of $X$.  Likewise, the following definitions for diameter and open $\varepsilon$-neighborhood of a projection generalize the corresponding notions for a subset.

\begin{definition}[{\cite[Proposition 2.16]{KuperbergWeaver}}]
\label{def:diameter}
Given a quantum metric $\mathbb{V}=\{\mathcal{V}_t\}_{t\in[0,\infty
)}$ on a von Neumann algebra $\mathcal{M}$, the \emph{diameter} of a nonzero projection $P$ in $\mathcal{M}$ is
\[\diam_\mathbb{V}(P)=\sup \big\{\dist_\mathbb{V}(Q,R)\mid Q(PAP\otimes \id)R\neq 0 \mbox{ for some }A\in \mathcal{B}(\mathcal{H})\big\}.\]
The diameter of the zero projection is defined to be zero.
The subscript will usually be omitted if there is no ambiguity regarding the quantum metric being used.
\end{definition}

We remark that \cite{KuperbergWeaver} only considers the case of $P$ being the identity of $\mathcal{M}$ in the definition above, so our notion is really a generalization of theirs.

\begin{definition}[{\cite[Proposition 2.17]{KuperbergWeaver}}]
\label{def:neighborhood}
Given a quantum metric $\mathbb{V}=\{\mathcal{V}_t\}_{t\in[0,\infty
)}$ on a von Neumann algebra $\mathcal{M}$, the \emph{open $\varepsilon$-neighborhood} of a projection $P$ in $\mathcal{M}$ is the projection
\[(P)_\varepsilon=\id-\bigvee\big\{Q\in \mathcal{M}\mid \dist_\mathbb{V}(P,Q)\geq \varepsilon\big\}.\]
\end{definition}

It is easy to see from \cite[Definition 2.14 (b)]{KuperbergWeaver} that $((P)_{\varepsilon})_\delta\leq (P)_{\varepsilon+\delta}$ for all $\varepsilon, \delta >0$ and projections $P$ in a quantum metric space $\mathcal{M}$.

Our next definition follows \cite{kornell2011quantum}.

\begin{definition}
A function $\phi\colon \mathcal{M}\to \mathcal{N}$ between two von Neumann algebras $\mathcal{M}$ and $\mathcal{N}$ is called a \emph{quantum function} if $\phi$ is a unital weak*-continuous $*$-homomorphism.
\end{definition}

Quantum functions generalize classical functions in the following way:  If $f\colon X\to Y$ is a classical function between sets, then $\phi_f\colon \ell_\infty(Y)\to \ell_\infty(X)$, defined by $\phi_f(g)=g\circ f$ for all $g\in \ell_\infty (Y)$, is the canonical quantum function associated to $f$ between the canonical von Neumann algebras associated to $Y$ and $X$, and is such that $\phi_f(\chi_{\{y\}})=\chi_{\{f^{-1}(y)\}}$ for all $y\in Y$. 

In what follows, we will often identify a von Neumann algebra $\mathcal{M}$ with $\mathcal{M}\otimes\id\subseteq \mathcal{M}\overline{\otimes}\mathcal{B}(\ell_2)$ and a quantum function $\phi\colon \mathcal{M}\to\mathcal{N}$ with $(\phi\otimes\id)\colon \mathcal{M}\overline{\otimes}\mathcal{B}(\ell_2)\to \mathcal{N}\overline{\otimes}\mathcal{B}(\ell_2)$.

Recall \cite[p. 17]{KuperbergWeaver} that two projections $P,Q \in \mathcal{M}\overline{\otimes} \mathcal{B}(\ell_2)$ are said to be \emph{unlinkable} if there exist $\tilde{P},\tilde{Q} \in \id \otimes \mathcal{B}(\ell_2)$ satisfying $P \le \tilde{P}$, $Q \le \tilde{Q}$ and $\tilde{P}\tilde{Q} = 0$; otherwise, they are said to be \emph{linkable}.
By \cite[Prop. 2.13]{KuperbergWeaver}, for a von Neumann algebra $\mathcal{M} \subseteq \mathcal{B}(H)$, two projections $P,Q \in \mathcal{M}\overline{\otimes} \mathcal{B}(\ell_2)$ are linkable if and only if there exists $A \in \mathcal{B}(H)$ such that $P(A \otimes \id)Q \not=0$.

Note that we may generalize these notions to non-projections:
two operators $S,T \in \mathcal{M}\overline{\otimes} \mathcal{B}(\ell_2)$ are said to be \emph{linkable} if and only if there exists $A \in \mathcal{B}(H)$ such that $S(A \otimes \id)T \not=0$; obviously, this is the same as saying that the source projection $\llbracket S \rrbracket$ of $S$ and the range projection $[T]$ of $T$ are linkable.
By the aforementioned characterization, two operators $S,T \in \mathcal{M}\overline{\otimes} \mathcal{B}(\ell_2)$ are not linkable (i.e. \emph{unlinkable}) if and only if there exist projections $\tilde{S},\tilde{T} \in \id \otimes \mathcal{B}(\ell_2)$ satisfying $\llbracket S \rrbracket \le \tilde{S}$, $[T] \le \tilde{T}$ and $\tilde{S}\tilde{T} = 0$.
Note that this is closely related to our definition of diameter of a projection:
For a projection $P$ in $\mathcal{M}$, its diameter is the supremum of the distances $\dist(Q,R)$ where $Q,R \in \mathcal{M}\overline{\otimes} \mathcal{B}(\ell_2)$ are projections such that  $Q(P\otimes \id)$ and $(P \otimes \id)R$ are linkable.

The following lemma shows that when $\phi$ is a quantum function, $\phi \otimes \id$ maps unlinkable pairs of operators to unlinkable pairs of operators.

\begin{lemma}\label{lemma:quantum-function-unlinkable}
Let $\phi : \mathcal{M} \to \mathcal{N}$ be a quantum function between von Neumann algebras. If $S,T \in \mathcal{M}\overline{\otimes} \mathcal{B}(\ell_2)$ are unlinkable operators, then $(\phi \otimes \id)S, (\phi \otimes \id)T \in \mathcal{N}\overline{\otimes} \mathcal{B}(\ell_2)$ are unlinkable as well.
\end{lemma}

\begin{proof}
As above, let $\tilde{S},\tilde{T} \in \id \otimes \mathcal{B}(\ell_2)$ be projections satisfying $\llbracket S \rrbracket \le \tilde{S}$, $[T] \le \tilde{T}$, and $\tilde{S}\tilde{T} = 0$.
Note that since $\phi \otimes \id$ is a $*$-homomorphism and thus preserves order, $(\phi \otimes \id)\tilde{S},(\phi \otimes \id)\tilde{T} \in \id \otimes \mathcal{B}(\ell_2)$ are projections and they satisfy
$(\phi\otimes \id)\llbracket S \rrbracket \le (\phi \otimes \id)\tilde{S}$, $(\phi\otimes \id)[T] \le (\phi \otimes \id)\tilde{T}$, and
$(\phi\otimes \id)\tilde{S}(\phi\otimes \id)\tilde{T} = 0$.
The desired result will follow once we prove that $(\phi\otimes \id)\llbracket S \rrbracket= \llbracket(\phi\otimes \id)S\rrbracket$ and $(\phi\otimes \id)[T] = [(\phi \otimes \id)T]$.

By our definition of quantum function, $\phi$ is a weak* to weak* continuous $*$-homomorphism.
By \cite[Prop. III.2.2.2]{Blackadar}, $\phi$ is normal.
Since the tensor product of normal completely positive contractions is again a normal positive contraction \cite[III.2.2.5]{Blackadar}, $\phi \otimes \id$ is normal. 
By \cite[Prop. III.2.2.2]{Blackadar} again,
$\phi \otimes \id$ is $\sigma$-strong to $\sigma$-strong continuous from 
$\mathcal{M}\overline{\otimes} \mathcal{B}(\ell_2)$ to
$\mathcal{N}\overline{\otimes} \mathcal{B}(\ell_2)$.
It is  well-known that the $\sigma$-strong and the strong topologies coincide on bounded sets \cite[I.3.1.4]{Blackadar}, so in particular it follows that $\phi \otimes \id$ maps bounded strongly convergent nets to bounded strongly convergent nets.

Now, it is known that $(S^*S)^\alpha \to \llbracket S \rrbracket$ strongly as $\alpha \to 0$ \cite[I.5.2.1]{Blackadar}.
Since $\{(S^*S)^\alpha\}_{\alpha\in(0,1)}$ is norm bounded (which can be easily shown using functional calculus),
we conclude that
\[
(\phi \otimes \id)\big( (S^*S)^\alpha \big) \xrightarrow[\alpha\to 0]{  } (\phi \otimes \id)\llbracket S\rrbracket
\]
strongly.
Since $(\phi \otimes \id)$ is a unital $*$-homomorphism we have $(\phi \otimes \id)\big( (S^*S)^\alpha \big) = \big( \big((\phi \otimes \id)S\big)^*\big((\phi \otimes \id)S\big) \big)^\alpha$ (again this can be easily shown using functional calculus). 
Now, 
\[
\big( \big((\phi \otimes \id)S\big)^*\big((\phi \otimes \id)S\big) \big)^\alpha \xrightarrow[\alpha\to 0]{  } \llbracket (\phi \otimes \id)S \rrbracket
\]
strongly, and therefore $\llbracket (\phi \otimes \id)S \rrbracket = (\phi \otimes \id)\llbracket S \rrbracket$.
The analogous conclusion for the range projection of $T$ follows from the fact that $[T]= \llbracket TT^* \rrbracket$ \cite[I.5.2.1]{Blackadar}.
\end{proof}

Some of our results will only apply to quantum metric spaces that are (operator) reflexive; we recall the definition below.

\begin{definition}[{\cite[Defns. 1.5 and 2.23]{KuperbergWeaver}}]
A subspace $\mathcal{V} \subseteq \mathcal{B}(H)$ is \emph{(operator) reflexive} if
$
\mathcal{V} = \big\{ B \in \mathcal{B}(H) \mid P\mathcal{V}Q =\{0\} \Rightarrow PBQ=0 \big\}
$
with $P$ and $Q$ ranging over projections in $\mathcal{B}(H)$.
A quantum metric $\mathbb{V}=\{\mathcal{V}_t\}_{t\in[0,\infty
)}$ is called \emph{reflexive} if $\mathcal{V}_t$ is reflexive for each $t\in[0,\infty)$.
\end{definition}

\section{Quantum moduli of expansion and compression}\label{sec:moduli}
In this section, we define coarse embeddings between quantum metric spaces using moduli and then show how this relates to the definitions of co-Lipschitz and co-isometric morphisms found in \cite{KuperbergWeaver}.

Recall that if $f\colon X\to Y$ is a map between metric spaces, we define its modulus of expansion $\omega_f$ by
\[
\omega_f(t)=\sup\big\{d_Y(f(x),f(y))\mid d_X(x,y)\leq t\big\}
\]
 and its modulus of compression $\rho_f$ by
\[
\rho_f(t)=\inf\big\{d_Y(f(x),f(y))\mid d_X(x,y)\geq t\big\}
\]
 for all $t\geq 0$.  We say that $f$ is \emph{expanding} if $\lim_{t\to\infty}\rho_f(t)=\infty$, and \emph{coarse} if $\omega_f(t)<\infty$ for all $t\geq 0$. 
We say that $f$ is a \emph{coarse embedding} if $f$ is both coarse and expanding. 

For our purposes, we will use  alternative versions of these moduli.
Let
\[
\tilde{\omega}_f(t)=\inf\big\{ d_X(x,y) \mid d_Y(f(x),f(y))\ge t\big\}
\]
and let
\[
\tilde{\rho}_f(t)=\sup\big\{ d_X(x,y) \mid d_Y(f(x),f(y))\le t\big\}.
\]
The following observation is surely well-known.

\begin{lemma}
\label{lem:coarse definition}
Let $f \colon X \to Y$ be a map between metric spaces. Then:
\begin{enumerate}[(a)]
\item  $f$ is coarse if and only if $\lim_{t\to\infty}\tilde{\omega}_f(t)=\infty$.
\item $f$ is expanding if and only if $\tilde{\rho}_f(t)<\infty$ for all $t \ge 0$.
\end{enumerate}
\end{lemma}

\begin{proof}
Note that $\tilde{\omega}$ is an increasing function, and therefore $\lim_{t\to\infty}\tilde{\omega}_f(t)=\infty$ if and only if $\tilde{\omega}_f$ is unbounded.

Suppose first that $f$ is not coarse so that by definition there exists $t\ge 0$ such that $\omega_f(t)=\infty$.
Then for each $n\in\N$ there exist $x_n,y_n \in X$ such that $d_Y(f(x_n),f(y_n)) \ge n$ and $d_X(x_n,y_n) \le t$.
This implies $\tilde{\omega}_f(n) \le t$, so $\tilde{\omega}_f$ is bounded above by $t$.

Suppose now that $\tilde{\omega}_f$ is bounded above by $t$.
Then for each $n\in\N$, there exist $x_n,y_n \in X$ such that $d_X(x_n,y_n) \le t+1$ while $d_Y(f(x_n),f(y_n)) \ge n$. This implies $\omega_f(t+1) = \infty$. That is, $f$ is not coarse.  This finishes the proof of (a), and the proof for (b) is analogous.
\end{proof}

Let us now define corresponding moduli for quantum functions.

\begin{definition}
Given a quantum function $\phi\colon \mathcal{M}\to \mathcal{N}$ between quantum metric spaces $\mathcal{M}$ and $\mathcal{N}$, we define $\tilde{\omega}_\phi$ and $\tilde{\rho}_\phi$ by
\[
\tilde{\omega}_\phi(t)=\inf\big\{\dist(\phi(P),\phi(Q)) \mid  \dist(P,Q) \geq t\big\}
\]
and
\[
\tilde{\rho}_\phi(t)= \sup\big\{ \diam( \phi(P) ) \mid  \diam(P) \le t\big\}
\]
for all $t\geq 0$, where $P,Q$ range over projections in $\mathcal{M}$.
\end{definition}

The next proposition shows that the moduli defined above generalize the classical moduli.

\begin{proposition}
\label{prop:moduli generalization}
Given metric spaces $X,Y$ and a function $f\colon X\to Y$,
$\tilde{\omega}_{\phi_f} = \tilde{\omega}_f$
and
$\tilde{\rho}_{\phi_f} = \tilde{\rho}_f$.
\end{proposition}

\begin{proof}
Let $P$ be a projection in $\ell_\infty(Y)$. Then $P = \chi_S$ for some $S \subseteq Y$,
and $\phi_f(P) = \chi_{f^{-1}[S]}$.
Therefore, for $t \ge 0$
\begin{multline*}
\tilde{\omega}_{\phi_f}(t)=\inf\{\dist(\chi_{f^{-1}[S]},\chi_{f^{-1}[T]}) \mid  \dist(\chi_S,\chi_T) \geq t\} \\
= \inf\{ d_X(x,y) \mid d_Y(f(x),f(y))\ge t\}
\end{multline*}
and
\begin{multline*}
\tilde{\rho}_{\phi_f}(t)= \sup\{ \diam( \chi_{f^{-1}[S]} ) \mid  \diam(\chi_S) \le t\} \\
= \sup\{ d_X(x,y) \mid d_Y(f(x),f(y))\le t\}.\qedhere
\end{multline*}

\end{proof}

Proposition \ref{prop:moduli generalization} and Lemma \ref{lem:coarse definition} justify the following definition.  Although it would perhaps be in better keeping with \cite{KuperbergWeaver} to use the terminology ``co-coarse'', there are two reasons we do not do this.  The first reason is that the inequalities involved concern only a quantum function $\phi$ and not its amplification $\phi \otimes \id$.  The second reason is that we are only exploring a notion of coarseness for functions between quantum metric spaces and not for operators inside of quantum metric spaces.

\begin{definition}
A quantum function $\phi\colon \mathcal{M}\to \mathcal{N}$ between two quantum metric spaces is called a \emph{(quantum) coarse embedding} if $\lim_{t\to\infty}\tilde{\omega}_\phi(t)=\infty$ and $\tilde{\rho}_\phi(t)<\infty$ for all $t\geq 0$.
\end{definition}

\begin{remark}
In \cite[Definition 2.27]{KuperbergWeaver}, a quantum function $\phi\colon \mathcal{M}\to \mathcal{N}$ is called a \emph{co-Lipschitz morphism} if there is some $C\geq 0$ such that
\[\dist(P,Q)\leq C\dist((\phi\otimes \id)(P),(\phi\otimes\id)(Q))\]
for all projections $P,Q \in \mathcal{M}\overline{\otimes}\mathcal{B}(\ell_2)$.  It is easily observed that if $\phi$ is a co-Lipschitz morphism, then $\tilde{\omega}_\phi (t)\geq t/C$ for all $t\geq 0$.  
\end{remark}

\begin{remark}\label{moduli-for-co-isometric-morphism}
Also in \cite[Definition 2.27]{KuperbergWeaver}, a quantum function $\phi\colon \mathcal{M}\to\mathcal{N}$ is called a \emph{co-isometric morphism} if it is surjective and
\[\dist(\tilde{P}, \tilde{Q})=\sup \big\{\dist(P,Q)\mid (\phi\otimes\id)(P)=\tilde{P}, (\phi\otimes\id)(Q)=\tilde{Q}\big\}\]
for all projections $\tilde{P},\tilde{Q}\in \mathcal{N}\overline{\otimes} \mathcal{B}(\ell_2)$.  If $\phi$ is a co-isometric morphism, then in particular, $\phi$ is a co-Lipschitz morphism with constant 1, and so $\tilde{\omega}_\phi(t)\geq t$ for all $t\geq 0$;  it may be shown that additionally $\tilde{\rho}_\phi(t)\leq t$ for all $t\geq 0$.  Indeed, if $P$ is a projection in $\mathcal{M}$, and $\tilde{Q}, \tilde{R}$ are projections in $\mathcal{N}\overline{\otimes}\mathcal{B}(\ell_2)$ such that $\tilde{Q}(\phi(P)\otimes \id)$ and $(\phi(P)\otimes \id)\tilde{R}$ are linkable, then since $\phi$ is a co-isometric morphism, by Lemma \ref{lemma:quantum-function-unlinkable},
\begin{multline*}
 \dist  (\tilde{Q},\tilde{R})
 = \sup\big\{ \dist(Q,R) \mid   (\phi \otimes \id)(Q) =  \tilde{Q}, (\phi \otimes \id)(R) = \tilde{R} \big\} \\
 \le
 \sup\big\{ \dist(Q,R) \mid  (\phi \otimes \id)\big(Q(P\otimes \id)\big) \text{ and } (\phi \otimes \id)\big((P \otimes \id)R\big) \text{ are linkable}   \big\}\\
 \le \sup \big\{\dist(Q,R)\mid Q(P\otimes \id) \text{ and } (P \otimes \id)R \text{ are linkable} \big\}\\
  = \diam(P). 
 \end{multline*}
Thus, $\diam(\phi(P))\leq \diam(P)$, and therefore $\tilde{\rho}_\phi(t)\leq t$.
\end{remark}

\section{Asymptotic dimension}\label{sec:asdim}
We will provide a definition of asymptotic dimension that can be applied generally to all quantum metric spaces.  Given the definitions that already exist for diameter and $\varepsilon$-neighborhood of a projection, we have chosen to base our generalization on Part 2 of \cite[Theorem 2.1.2]{Bedlewo}.  We do not explore generalizations of the equivalent formulations of asymptotic dimension found in \cite[Theorem 2.1.2]{Bedlewo}.
\begin{definition}
Let $\mathcal{M}$ be a quantum metric space and $\mathscr{P}$ a family of projections in $\mathcal{M}$.  We say that $\mathscr{P}$ is a \emph{cover} for $\mathcal{M}$ if $\id=\bigvee_{P\in\mathscr{P}}P$. 
We say that $\mathscr{P}$ is \emph{$r$-disjoint} if $(P)_r(Q)_r = 0$ for each $P,Q \in \mathscr{P}$ with $P \not=Q$.   We say that $\mathscr{P}$ is \emph{uniformly bounded by $R$} if $\sup_{P\in\mathscr{P}}\diam(P)\leq R$,
and that $\mathscr{P}$ is \emph{uniformly bounded} if it is uniformly bounded by some $R>0$.
\end{definition}

\begin{definition}
\label{def:asdim}
Let $\mathcal{M}$ be a quantum metric space, and let $n \in \N\cup \{0\}$. We say that $\mathcal{M}$ has \emph{asymptotic dimension less than or equal to $n$}, written as $\asdim(\mathcal{M}) \le n$, if
for every $r>0$ there exist uniformly bounded, $r$-disjoint families of projections $\mathscr{P}^0,\mathscr{P}^1, \dotsc, \mathscr{P}^n$  such that $\bigcup_{i=0}^n \mathscr{P}^i$ is a cover for  $\mathcal{M}$.  We say that $\mathcal{M}$ has \emph{asymptotic dimension equal to $n$}, written as $\asdim(\mathcal{M})=n$, if $n=\min\{m\in \mathbb{N}\cup \{0\}\mid \asdim(\mathcal{M})\leq m\}$.
\end{definition}

\begin{remark}
Let $(X,d)$ be a metric space, and consider the von Neumann algebra $\ell_\infty(X)$ endowed with the canonical quantum metric induced by $d$.
It is clear that $\asdim(X) = \asdim(\ell_\infty(X))$, since the projections in $\ell_\infty(X)$ are precisely the indicator functions of subsets of $X$.
\end{remark}

For classical metric spaces, coarse embeddings are the natural morphisms that preserve asymptotic dimension because for any $r>0$ they map every $R$-disjoint, uniformly bounded family of sets to an $r$-disjoint, uniformly bounded family of sets whenever $R>0$ is large enough.  This follows easily from the definition of coarse embedding using the moduli of expansion and compression.  We show that the same holds true for coarse embeddings between quantum metric spaces.


\begin{lemma}
\label{lemma:boundedness-under-mappings}
Let $\phi\colon \mathcal{M}\to \mathcal{N}$ be a quantum function between quantum metric spaces.  Then for any projection $P\in \mathcal{M}$,
\[\diam(\phi(P))\leq \tilde{\rho}_\phi(\diam(P)).\]
In particular, $\phi$ maps a family of projections uniformly bounded by $R$ to a family of projections uniformly bounded by $\tilde{\rho}_\phi(R)$.
\end{lemma}
\begin{proof}
\[\diam(\phi(P))
\leq \sup \{\diam(\phi(Q))\mid \diam(Q)\leq \diam(P)\}
=\tilde{\rho}_\phi(\diam(P)).\qedhere\]
\end{proof}

\begin{lemma}\label{lemma:r-disjointness-under-mappings}
Let $\phi\colon \mathcal{M}\to \mathcal{N}$ be a quantum function between quantum metric spaces, and let $r>0$.
Then for any projection $P \in \mathcal{M}$,
\[
\big(\phi(P) \big)_{\tilde{\omega}_\phi(r)} \le \phi\big( (P)_r \big). 
\]
In particular, $\phi$ maps an $r$-disjoint family of projections to an $\tilde{\omega}_\phi(r)$-disjoint family of projections.
\end{lemma}

\begin{proof}
Since $\phi$ is a quantum function and $\dist(P,Q) \ge r$ implies $\dist(\phi(P),\phi(Q)) \ge \tilde{\omega}_\phi(r)$, we have 
\begin{align*}
\phi\big( (P)_r \big) &= \phi\big( \id_{\mathcal{M}} -  \bigvee \big\{ Q\in\mathcal{M} \mid \dist(P,Q) \ge r \big\} \big) \\
&= \id_{\mathcal{N}} -  \bigvee \big\{ \phi(Q) \mid \dist(P,Q) \ge r \big\} \\
&\ge \id_{\mathcal{N}} -  \bigvee \big\{ \phi(Q) \mid \dist(\phi(P),\phi(Q)) \ge \tilde{\omega}_\phi(r) \big\} \\
&\ge \id_{\mathcal{N}} -  \bigvee \big\{ R\in \mathcal{N} \mid \dist(\phi(P),R) 
\ge \tilde{\omega}_\phi(r) \big\} \\
&= \big(\phi(P) \big)_{\tilde{\omega}_\phi(r)}.\qedhere
\end{align*}
\end{proof}

The next theorem follows immediately.  Note that a quantum function is unital and so it maps covers to covers.

\begin{theorem}\label{thm:coarse-embeddings}
Let $\phi\colon \mathcal{M}\to \mathcal{N}$ be a quantum coarse embedding between quantum metric spaces.
Then $\asdim(\mathcal{N}) \le \asdim(\mathcal{M})$.
\end{theorem}

As a consequence of Theorem \ref{thm:coarse-embeddings}, asymptotic dimension plays well with the quotient \cite[Definition 2.35]{KuperbergWeaver} and direct sum \cite[Definition 2.32 (b)]{KuperbergWeaver} constructions for quantum metric spaces.  Compare this to the corresponding results on subspaces and (disjoint) unions of classical metric spaces found in \cite[Proposition 2.2.6]{Bedlewo} and \cite[Corollary 2.3.3]{Bedlewo}.
Note also that these results are related to some of the conditions required of an abstract dimension theory for a class of $C^*$-algebras from \cite{Thiel}.

\begin{corollary}
Let $\mathcal{M}$ and $\mathcal{N}$ be quantum metric spaces.
\begin{enumerate}[(a)]
\item\label{permanence:quotient}
If $\mathcal{N}$ is a metric quotient of $\mathcal{M}$, then $\asdim(\mathcal{N}) \le \asdim(\mathcal{M})$.
\item\label{permanence:sum}
$\asdim( \mathcal{M} \oplus \mathcal{N} ) = \max \{ \asdim(\mathcal{M}), \asdim(\mathcal{N}) \}$.
\end{enumerate}
\end{corollary}

\begin{proof}
\eqref{permanence:quotient}: This follows immediately from Theorem \ref{thm:coarse-embeddings} and Remark \ref{moduli-for-co-isometric-morphism} of this paper and \cite[Corollary 2.37]{KuperbergWeaver}.

\eqref{permanence:sum}: 
Since each of $\mathcal{M}$ and $\mathcal{N}$ is a metric quotient of $\mathcal{M} \oplus \mathcal{N} $, we have $\asdim( \mathcal{M} \oplus \mathcal{N} ) \geq \max \{ \asdim(\mathcal{M}), \asdim(\mathcal{N}) \}$ from part \eqref{permanence:quotient}.
Now let $n = \max \{ \asdim(\mathcal{M}), \asdim(\mathcal{N}) \}$ and take any $r>0$.
By Definition \ref{def:asdim}, there exist uniformly bounded, $r$-disjoint families of projections $\mathscr{P}^0,\mathscr{P}^1, \dotsc, \mathscr{P}^n$  such that $\bigcup_{i=0}^n \mathscr{P}^i$ is a cover for  $\mathcal{M}$, and there also exist uniformly bounded, $r$-disjoint families of projections $\mathscr{Q}^0,\mathscr{Q}^1, \dotsc, \mathscr{Q}^n$  such that $\bigcup_{i=0}^n \mathscr{Q}^i$ is a cover for  $\mathcal{N}$.
For $1\leq j\leq n$, define $\mathscr{R}^j = \{ P \oplus 0 \mid P \in \mathscr{P}^j\} \cup \{ 0 \oplus Q \mid Q \in \mathscr{Q}^j\}$.
It is clear that each $\mathscr{R}^j$ is uniformly bounded, and moreover the union $\bigcup_{i=0}^n \mathscr{R}^i$ is a cover for $\mathcal{M} \oplus \mathcal{N}$. Additionally, each family $\mathscr{R}^j$ is $r$-disjoint, since $(P \oplus 0)_r = (P)_r \oplus 0$ and $(0 \oplus Q)_r = 0 \oplus (Q)_r$, by \cite[Proposition 2.34]{KuperbergWeaver}.
Therefore, $\asdim(\mathcal{M} \oplus \mathcal{N}) \le n$.
\end{proof}

An analog of the result concerning asymptotic dimension of possibly nondisjoint unions of classical metric spaces \cite[Corollary 2.3.3]{Bedlewo} can also be established, at least for reflexive quantum metric spaces.  In a reflexive quantum metric space $\mathcal{M}$, diameters of projections in $\mathcal{M}$ may be computed using only projections in $\mathcal{M}$.  

\begin{lemma}
\label{lem:reflexive_diameter}
Let $(\mathcal{M}, \{\mathcal{V}_t\}_{t\in [0,\infty)})$ be a reflexive quantum metric space and let $P$ be a nonzero projection in $\mathcal{M}$.  Then
\begin{align*}
\diam (P)&=\sup \{\dist (Q,R)\mid QPAPR\neq 0 \mbox{ for some }A\in \mathcal{B}(\mathcal{H})\}\\
&=\sup\{\dist (Q,R)\mid QP,RP\neq 0\}
\end{align*}
\end{lemma}

\begin{proof}
It is clear from Definition \ref{def:diameter} that 
\[\diam (P)\geq\sup \{\dist (Q,R)\mid QPAPR\neq 0 \mbox{ for some }A\in \mathcal{B}(\mathcal{H})\}.\]
The result is then trivial when $\diam (P)=0$.  So suppose $\diam (P)\neq 0$ and take
any $0<\varepsilon<\diam (P)$.  Let $Q,R$ be any projections in $\mathcal{M}\otimes \mathcal{B}(\mathcal{H})$ such that $\dist (Q,R)>\diam (P)-\varepsilon$ while $Q(PAP\otimes \id)R\neq 0$ for some $A\in \mathcal{B}(\mathcal{H})$.  By \cite[Proposition 2.10]{KuperbergWeaver}, $PAP\notin \mathcal{V}_{\diam (P)-\varepsilon}$.  But then by \cite[Proposition 2.24]{KuperbergWeaver}, there exist projections $Q',R'\in \mathcal{M}$ such that $\dist (Q',R')\geq \diam (P)-\varepsilon$ while $QPAPR\neq 0$.  As $\varepsilon>0$ was arbitrary, it follows that 
\[\diam (P)\leq\sup \{\dist (Q,R)\mid QPAPR\neq 0 \mbox{ for some }A\in \mathcal{B}(\mathcal{H})\}.\qedhere\]
\end{proof}

\begin{lemma}\label{lemma:r-saturated-union}
Let $\mathcal{M}$ be a reflexive quantum metric space, let $P,Q,R$ be projections in $\mathcal{M}$, and take any $r,s>0$. Then:
\begin{enumerate}[(a)]
\item $(P)_rQ \not=0 \iff  \dist(P,Q) < r$.
\item $\dist(Q,R) \le \dist(Q,P) +\dist(R,P) + \diam(P)$.
\item $\diam\big( (P)_r \big) \le \diam(P)+2r$.
\item $\diam( P \vee Q ) \le \dist(P,Q)+\diam(P)+\diam(Q)$.
\item $(P)_r(Q)_s\neq 0\implies \diam (P\vee Q)\leq \diam(P)+\diam(Q)+2(r+s)$.
\end{enumerate}
\end{lemma}

\begin{proof}
(a) The implication $\implies$ follows immediately from Definition \ref{def:neighborhood}.
So suppose $Q$ is such that $\dist(P,Q) < r$ and furthermore, that $(P)_rQ = 0$.  Then
\[
Q \le \id - (P)_r = \bigvee\{Q'\in \mathcal{M} \mid \dist(P,Q') \ge r\} 
\]
and thus
\[
r > \dist(P,Q) \ge \dist(P, \id-(P)_r) = \inf\{ \dist(P,Q') \mid \dist(P,Q') \ge r \} \ge r,
\]
a contradiction.  Therefore $(P)_rQ \neq 0$ if $\dist(P,Q)<r$.

(b) By the remark after Definition \ref{def:distance}, with $S$ ranging over projections in $\mathcal{M}$,
\begin{align*}
\dist(Q&,R)\\
&\leq \dist (Q,P)+\sup \{\dist (S,R)\mid PS\neq 0\}\\
&\leq \dist (Q,P)+\dist (R,P)+\sup\{\dist (S,S')\mid PS\neq 0, PS'\neq 0\}\\
&\leq \dist (Q,P)+\dist (R,P)+\diam(P),
\end{align*}
where the last inequality follows from the fact that $SPAPS'\neq 0$ for some $A\in\mathcal{B}(\mathcal{H})$ whenever $SP, PS'\neq 0$.

(c) Suppose $S,S'$ are projections in $\mathcal{M}$ such that $(P)_rS\neq 0$ and $(P)_rS'\neq 0$.  By parts (a) and (b), this means
$\dist (S,S')\leq \diam (P)+2r$.  As $S,S'$ were arbitrary, Lemma \ref{lem:reflexive_diameter} implies that $\diam ((P)_r)\leq \diam (P)+2r$.

(d) Suppose $S,S'$ are projections in $\mathcal{M}$ such that $S(P\vee Q)\neq 0$ and $S'(P\vee Q) \not=0$.  If $SP \not=0$ and $S'P\not=0$, then $\dist(S,S') \le \diam(P)$.  If $SQ \not=0$ and $S'Q\not=0$, then $\dist(S,S') \le \diam(Q)$.  Finally, if $SP \not=0$ and $S'Q\not=0$, then by part (b),
\begin{align*}
\dist(S,S') &\le \dist(S,Q) + \dist (Q,S')+\diam(Q)\\
&\le \dist (S,P)+\dist (P,Q)+ \diam (P) + \diam(Q)\\
&= \dist (P,Q)+\diam(P) + \diam(Q).
\end{align*}
The same inequality holds if $SQ\neq 0$ and $S'P\neq 0$.  As $S,S'$ were arbitrary, Lemma \ref{lem:reflexive_diameter} implies that $\diam(P\vee Q)\leq \dist(P,Q) +\diam(P) + \diam(Q)$.

(e) By parts (c) and (d)
\begin{align*}
    \diam (P\vee Q)&\leq \diam ((P)_r\vee (Q)_s)\\
    &\leq \dist ((P)_r,(Q)_s)+\diam ((P)_r)+\diam((Q)_s)\\
    &\leq \diam (P)+\diam (Q)+2(r+s).\qedhere
\end{align*}
\end{proof}

The following is a direct adaptation of  \cite[Prop. 2.3.1]{Bedlewo} for reflexive quantum metric spaces.

\begin{proposition}
\label{prop:r-saturated union}
Let $\mathcal{M}$ be a reflexive quantum metric space, and let $\mathscr{P}, \mathscr{Q}$ be families of projections in $\mathcal{M}$.  Fix $r>0$ and for each $Q\in \mathscr{Q}$, let 
\[
    \mathscr{P}_Q=\{P\in \mathscr{P}\mid (P)_r(Q)_r \not=0 \}\mbox{ and  }
    P_Q=\bigvee_{P\in\mathscr{P}_Q}P.\]
Suppose that $\mathscr{P}$ is $r$-disjoint and $R$-bounded with $R > r$,
and $\mathscr{Q}$ is $7R$-disjoint and $D$-bounded.
Then $\mathscr{Q} \cup_r \mathscr{P}$ is $r$-disjoint and $(D + 2(R+D+4r))$-bounded, where
\[\mathscr{Q} \cup_r \mathscr{P} = \left\{ Q \vee P_Q  \mid Q \in \mathscr{Q}  \right\}
 \cup \big\{ P \in \mathscr{P} \mid (P)_r(Q)_r = 0 \text{ for all } Q \in \mathscr{Q}  \big\}.\]
\end{proposition}

\begin{proof}
Fix $Q \in \mathscr{Q}$. By Lemmas \ref{lem:reflexive_diameter} and \ref{lemma:r-saturated-union} (b) and (e),
\begin{align*}
\diam(P_Q)
&=\sup \{\dist(S,S')\mid SP_Q, S'P_Q\neq 0 \}\\
&\leq \sup \{\dist(S,Q)+\dist(S',Q)\mid SP_Q, S'P_Q\neq 0\}+\diam(Q)\\
&\leq 2\sup_{P\in\mathscr{P}_Q}\{\diam(P\vee Q)\}+\diam(Q)\\
&\leq 2\left(\sup_{P\in\mathscr{P}_Q}\{\diam (P)\}+\diam(Q)+4r\right)+\diam(Q)\\
&\leq 2(R+D+4r)+D.
\end{align*}
Thus, the bound on the diameter of $Q\vee P_Q$ and hence the entire family $\mathscr{P}\cup_r\mathscr{Q}$ is shown.

We now show that $\mathscr{Q} \cup_r \mathscr{P} $ is $r$-disjoint.
If we take two elements of $\mathscr{Q} \cup_r \mathscr{P} $ coming from $\mathscr{P}$, then they are $r$-disjoint by assumption.
Using the fact that $(R\vee S)_r=(R)_r\vee (S)_r$ for all projections $R$ and $S$, it is also clear that any two elements such that one is of the form $Q \vee P_Q$ and the other is in $\big\{ P \in \mathscr{P} \mid (P)_r(Q)_r = 0 \text{ for all } Q \in \mathscr{Q}  \big\}$ will be $r$-disjoint.
The only remaining case is to consider two elements of the form
$Q \vee P_Q$ and $Q' \vee P_{Q'}$, 
where $Q,Q' \in \mathscr{Q}$ are distinct.  Note that in this case $(Q)_r(Q')_r=0$.  Consider $P,P'$ such that $P\in \mathscr{P}_Q$ and $P'\in \mathscr{P}_{Q'}$.  If $(P)_r(Q')_r\neq 0$, then by Lemma \ref{lemma:r-saturated-union} (a), (b), and (c),
\begin{multline*}
\dist (Q,Q')\leq \dist (Q,(P)_r)+\dist ((P)_r,Q')+\diam((P)_r)\\
<2r+\diam(P)+2r<5R.
\end{multline*}  By Lemma \ref{lemma:r-saturated-union} (a), this implies $(Q)_{5R}Q'\neq 0$, a contradiction.  Thus $(P)_r(Q')_r=0$ and similarly $(P')_r(Q)_r=0$.  And if $(P)_r(P')_r\neq 0$, then by Lemma \ref{lemma:r-saturated-union} (a), (b), (c), and (d),
\begin{multline*}
\dist (Q, (Q')_r)\leq \dist(Q, (P)_r\vee (P')_r)+\diam ((P)_r\vee (P')_r)\\\leq r+ \diam (P)+2r +\diam (P')+2r<7R.
\end{multline*}
By Lemma \ref{lemma:r-saturated-union} (a), this implies $(Q)_{7R}(Q')_R\neq 0$, a contradiction.  Thus $(P)_r(P')_r=0$.  As $P$, $P'$ were arbitrary, it follows that $(Q\vee P_Q)_r(Q'\vee P_{Q'})_r=0$.  Therefore $\mathscr{Q} \cup_r \mathscr{P} $ is $r$-disjoint.
\end{proof}

We can now prove the following theorem which provides a bound on the asymptotic dimension of ``nondisjoint unions'' of quantum metric spaces.  Compare this to \cite[Corollary 2.3.3]{Bedlewo}.

\begin{theorem}
\label{thm:nondisjoint_union}
Let $\mathcal{M}$ be a reflexive quantum metric space. Suppose that $\mathcal{N}_1$ and $\mathcal{N}_2$ are metric quotients of $\mathcal{M}$, corresponding to central projections $R_1$ and $R_2$, respectively.
If $R_1 \vee R_2 = \id$, then \[\asdim(\mathcal{M}) \le \max \{ \asdim(\mathcal{N}_1), \asdim(\mathcal{N}_2) \}.\]
(Note that, in particular, this includes the case $\mathcal{M} = \mathcal{N}_1 \oplus \mathcal{N}_2$). 
\end{theorem}

\begin{proof}
Let $n=\max \{ \asdim(\mathcal{N}_1), \asdim(\mathcal{N}_2) \}$ and fix $r>0$.
Take $n+1$ uniformly bounded, $r$-disjoint families of projections $\mathscr{P}^0,\mathscr{P}^1, \dotsc, \mathscr{P}^n$ in $\mathcal{N}_1$ such that $\bigcup_{i=0}^n \mathscr{P}^i$ is a cover for  $\mathcal{N}_1$ and let $R>r$ be a uniform diameter bound for $\bigcup_{i=0}^n \mathscr{P}^i$.
Now take $n+1$ uniformly bounded, $7R$-disjoint families of projections $\mathscr{Q}^0,\mathscr{Q}^1, \dotsc, \mathscr{Q}^n$ in $\mathcal{N}_2$ such that $\bigcup_{i=0}^n \mathscr{Q}^i$ is a cover for  $\mathcal{N}_2$ and let $D>0$ be a uniform diameter bound for $\bigcup_{i=0}^n \mathscr{Q}^i$.  By viewing a projection $P$ in $\mathcal{N}_1$ as the projection $P\oplus 0\in R_1\mathcal{M}\oplus (\id-R_1)\mathcal{M}\cong \mathcal{M}$ and a projection $Q\in \mathcal{N}_2$ as the projection $Q\oplus 0\in R_2\mathcal{M}\oplus (\id-R_2)\mathcal{M}\cong \mathcal{M}$, it follows from \cite[Proposition 2.34]{KuperbergWeaver} that the families $\mathscr{P}^0,\mathscr{P}^1, \dotsc, \mathscr{P}^n$ and $\mathscr{Q}^0,\mathscr{Q}^1, \dotsc, \mathscr{Q}^n$ have the same bounds and disjointedness when viewed as families of projections in $\mathcal{M}$.  
Thus, for each $0\leq j\leq n$, the families $\mathscr{R}^j = \mathscr{Q}^j \cup_r \mathscr{P}^j$ in $\mathcal{M}$ are $r$-disjoint and uniformly bounded by Proposition \ref{prop:r-saturated union}.  And since $R_1 \vee R_2 = \id$, it follows that $\bigcup_{i=0}^n \mathscr{R}^i$ is a cover for $\mathcal{M}$.  Therefore $\asdim (\mathcal{M})\leq n$.
\end{proof}

\begin{remark}
In order to prove the general quantum analog of \cite[Corollary 2.3.3]{Bedlewo} found in Theorem \ref{thm:nondisjoint_union}, we had to make the assumption that the quantum metric space is reflexive.  Our proof follows \cite{Bedlewo} rather closely and relies on the ability to place an upper bound on the diameter of a neighborhood of a projection in terms of the diameter of the projection itself.  This bound is found in Part (c) of Lemma \ref{lemma:r-saturated-union}, which is the first place we use the reflexivity assumption.  We do not know whether the reflexivity assumption can be dropped in the statement of Theorem \ref{thm:nondisjoint_union}, but if it can, we expect a method different from that found in \cite{Bedlewo} would be needed to prove it.
\end{remark}

\section{Asymptotic dimension and quantum expanders}\label{sec:expanders}
In this section, we will show that a quantum metric space equi-coarsely containing a sequence of classical expanders (or more generally, a sequence of reflexive quantum expanders) has infinite asymptotic dimension.  This is a generalization of \cite[Sec. 2.3]{Dranishnikov-Sapir} which shows that even for general quantum metric spaces, information about their large-scale structure can be inferred from their bounded metric subspaces.  While this statement is quite believable in light of \cite{Dranishnikov-Sapir} and Theorem \ref{thm:coarse-embeddings}, it is not obvious at all that a quantum metric space should coarsely contain a classical metric space of infinite asymptotic dimension even though it equi-coarsely contains a sequence of expander graphs!  To prove the statement, we establish a quantum version of a vertex-isoperimetric inequality for expanders from a known edge-isoperimetric inequality.

In what follows, we denote the space of $n\times n$ matrices with complex entries by $M_n$.  The $n\times n$ identity matrix will be denoted by $I_n$.  The Hilbert-Schmidt norm for matrices will be denoted by $\|\cdot\|_{\HS}$, the trace norm will be denoted by $\|\cdot\|_1$, and the operator norm will be denoted by $\|\cdot\|_\infty$. We will use the initialization CPTP for a map $\Phi\colon M_n\to M_n$ to indicate that $\Phi$ is completely positive and trace-preserving.  Given a completely positive map $\Phi\colon M_n\to M_n$, there exist by Choi's theorem \cite{Choi} matrices $K_1, K_2,\dots , K_N \in M_n$ such that $\Phi(X)=\sum_{j=1}^N K_jX K_j^*$ for all matrices $X\in M_n$.  If $\Phi$ is additionally trace-preserving, it may be shown also that $\sum_{i=1}^N K_i^* K_i=I_n$.  It is then possible to define a quantum metric $\mathbb{V}=\{\mathcal{V}_t\}_{t\in[0,\infty )}$ on $M_n$ by $\mathcal{V}_0=\mathbb{C}\cdot I_n$, $\mathcal{V}_1=\spa\{K_j^* K_i\}_{1\leq i,j\leq N}$, and $\mathcal{V}_t=\mathcal{V}_1^{\lfloor t\rfloor }$ for $t>0$ \cite[Sec. 3.2]{KuperbergWeaver}. 
There are good information-theoretical reasons \cite{Duan-Severini-Winter,weaver2015quantum} and metric reasons \cite{KuperbergWeaver} for regarding a CPTP map $\Phi$ (or rather, the operator system $\mathcal{V}_1$) as a quantum analog of a combinatorial graph and the quantum metric $\mathbb{V}$ a quantum analog of a graph metric.  By an abuse of language, the terminology ``quantum graph'' will be used for any of $\Phi$, $\mathcal{V}_1$, and $(M_n, \mathbb{V})$.

\begin{definition}
Given $\delta,\varepsilon, t>0$ and $n\in \mathbb{N}$, a quantum metric on $M_n$ is said to satisfy a \emph{$(\delta,\varepsilon,t)$-isoperimetric inequality} if 
\[
\rank\big( (P)_\delta \big) \ge (1+\varepsilon) \rank(P)
\]
for all projections $P \in M_n$ such that $\diam(P) \le t$.
\end{definition}

\begin{remark}
\label{remark:isoperimetric-repeat}
By Lemma \ref{lemma:r-saturated-union} (c), if a reflexive quantum metric on $M_n$ satisfies a $(\delta,\varepsilon,t)$-isoperimetric inequality, it follows from repeated applications of it that, given any $m\in\N$,
\[
\rank\big( (P)_{m \delta} \big) \ge (1+\varepsilon)^m \rank(P)
\] for all projections $P \in M_n$ such that $\diam(P) + 2m\delta \le t$.

\end{remark}

\begin{definition}
Given a family of quantum coarse embeddings $\{\phi_\alpha \colon \mathcal{M} \to \mathcal{N}_\alpha\}$, we say that the family is \emph{equi-coarse} if there exist functions $f,g$
satisfying $\lim_{t\to\infty} f(t) = \infty$ and $g(t)<\infty$ for all $t\geq 0$ such that for for each $t\geq 0$ and each $\alpha$, $f(t) \le \tilde{\omega}_{\phi_\alpha}(t)$ and $\tilde{\rho}_{\phi_\alpha}(t) \le g(t)$.
By analogy with the classical setting, in this case we say that $\mathcal{M}$ equi-coarsely contains the family $\{\mathcal{N}_\alpha\}$.
\end{definition}

The strategy of proof in the next Proposition is based on \cite[Thm. 2.9]{Dranishnikov-Sapir}.

\begin{proposition}
\label{prop:equi_coarse_infinite_dimension}
Let $\mathcal{M}$ be a quantum metric space, and fix $\delta, \varepsilon >0$.  Suppose that $\{(M_{n_t}, \mathbb{V}_t)\}_{t>0}$ is a family of reflexive quantum metric spaces and $\{\phi_t \colon \mathcal{M} \to M_{n_t}\}_{t>0}$ is an equi-coarse family of quantum coarse embeddings. If $M_{n_t}$ satisfies a $(\delta,\varepsilon,t)$-isoperimetric inequality for every $t>0$, then $\asdim(\mathcal{M}) = \infty$. 
\end{proposition}

\begin{proof}
Suppose $\mathcal{M}$ has finite asymptotic dimension $n$, and take any $m\in\N$ such that $(1+\varepsilon )^m -1 > n$.  Let $f,g$ be functions satisfying $\lim_{r\to\infty} f(r) = \infty$ and $g(r)<\infty$ for all $r\geq 0$ such that $f(r) \le \tilde{\omega}_{\phi_t}(r)$ and $\tilde{\rho}_{\phi_t}(r) \le g(r)$ for all $t,r >0$; and pick $r>0$ such that $f(r)>m\delta$.
Let  $\mathscr{P}^0,\mathscr{P}^1, \dotsc, \mathscr{P}^n$ be uniformly bounded $r$-disjoint families of projections in $\mathcal{M}$ such that $\bigcup_{j=0}^n \mathscr{P}^j$ is a cover for  $\mathcal{M}$ and let $d$ be such that $\diam(P) \le d$ for every $P \in \bigcup_{j=0}^n \mathscr{P}^j$.  Finally, let $t=2m\delta +g(d)$.
It follows from Lemmas \ref{lemma:boundedness-under-mappings} and \ref{lemma:r-disjointness-under-mappings} that for each $0\leq j\leq n$, the families
$\mathscr{Q}^j = \{ \phi_t(P) \mid P \in \mathscr{P}^j \}$ are $f(r)$-disjoint (and therefore $m\delta$-disjoint) and uniformly bounded by $g(d)$, and $\bigcup_{j=0}^n\mathscr{Q}^j$  is a cover for $M_{n_t}$ since $\bigcup_{j=0}^n \mathscr{P}^j$ is a cover for $\mathcal{M}$.
Thus, the $(\delta,\varepsilon,2m\delta +g(d))$-isoperimetric inequality for $M_{n_t}$ implies by Lemma \ref{lemma:r-saturated-union} (c) and Remark \ref{remark:isoperimetric-repeat} that for each $0\leq j\leq n$,
\[
n_t=\rank(I_{n_t}) \ge \sum_{Q \in \mathscr{Q}^j} \rank( (Q)_{m\delta}  ) \ge (1+\varepsilon)^m \sum_{Q \in \mathscr{Q}^j}  \rank(Q),
\]
and adding over $j$ yields
\[
(n+1) \cdot n_t \ge (1+\varepsilon)^m \sum_{j=0}^n \sum_{Q \in \mathscr{Q}^j}  \rank(Q) \ge (1+\varepsilon)^m  \cdot n_t,
\]
where the last inequality follows from the fact that $\bigcup_{j=0}^n \mathscr{Q}^j$ is a cover for $M_{n_t}$.
This implies that $n \ge (1+\varepsilon)^m-1>n$, a contradiction.  Therefore $\asdim(\mathcal{M}) = \infty$.
\end{proof}

We will show that a sequence of reflexive quantum expanders (which includes the case of classical expanders) satisfies the isoperimetric inequality condition found in Proposition \ref{prop:equi_coarse_infinite_dimension}.  We first recall the definition of quantum expander sequence and an associated Cheeger-type inequality below.

\begin{definition}[\cite{Pisier-Quantum-Expanders}]
Given $0 < \varepsilon<1$ and $n\in\mathbb{N}$, a CPTP map $\Phi \colon M_n \to M_n$ is said to have an \emph{$\varepsilon$-spectral gap} if 
\[
\n{ \Phi(X) - \tfrac{1}{n}\tr(X) I_n }_{\HS} \le (1-\varepsilon) \n{ X - \tfrac{1}{n}\tr(X) I_n }_{\HS}
\]
for all $X \in M_n$.
\end{definition}
\begin{definition}[\cite{Pisier-Quantum-Expanders}]
\label{def:expander}
A CPTP map  $\Phi \colon M_n \to M_n$ is called a \emph{$d$-regular $\varepsilon$-quantum expander} if $\Phi$ has an $\varepsilon$-spectral gap and there exist unitaries $U_1,\dotsc,U_d \in M_n$ such that $\Phi(X) = \frac{1}{d}\sum_{j=1}^d U_jXU_j^*$ for each $X \in M_n$.  
A sequence of CPTP maps $\{\Phi_m \colon M_{n_m} \to M_{n_m}\}$ is called a \emph{sequence of $d$-regular $\varepsilon$-quantum expanders} if $\Phi_m$ is a $d$-regular $\varepsilon$-quantum expander for each $m\in \mathbb{N}$ and $n_m \to \infty$ as $m\to\infty$.
\end{definition}

The following is just a restatement of \cite[Lemma 20]{TKRWV}, which can be described as a quantum Cheeger inequality.

\begin{lemma}\label{lemma-Cheeger}
Let $\Phi \colon M_n \to M_n$ be a CPTP unital map with an $\varepsilon$-spectral gap. Then 
\[
\frac{\tr\big( (I_n-P) \Phi^*\Phi(P) \big)}{\tr(P)}\geq (1-\varepsilon)/2
\]
for all projections $P\in M_n$ such that $0<\rank (P)\leq n/2$.
\end{lemma}

\begin{remark}
\label{remark:Cheeger}
The expression appearing in the preceding lemma can be rewritten in terms of the inner product associated to the Hilbert-Schmidt norm.  Indeed,
\begin{multline*}
\tr\big( (I_n-P) \Phi^*\Phi(P) \big) = \tr\big( (I_n-P)^* \Phi^*\Phi(P) \big) \\
= \pair{\Phi^*\Phi(P)}{I_n-P}_{\HS} = \pair{\Phi(P)}{\Phi(I_n-P)}_{\HS}.
\end{multline*}
Thus, Lemma \ref{lemma-Cheeger} says that $\Phi$ maps orthogonal pairs $P,I_n-P$ to nonorthogonal pairs in a uniform way. 
\end{remark}

The next result says that if the rank of a projection inside an expander is small, then any large neighborhood of the projection has strictly larger rank than the projection itself.

\begin{proposition}
\label{prop:expander_rank_inequality}
Let $\Phi \colon M_n \to M_n$ be a $d$-regular $\varepsilon$-quantum expander.
Then for any $\delta>1$,
\[
\rank\big( (P)_\delta \big) \ge (1+\varepsilon') \rank(P)
\]
whenever $P\in M_n$ is a projection such that $\rank(P) \le n/2$,
where the quantum metric on $M_n$ is the one induced by $\Phi$ and {$\varepsilon'=(1-\varepsilon)/2$}.
\end{proposition}

\begin{proof}
Let $P\in M_n$ be a projection such that $\rank (P)\leq n/2$.  We will show that if $Q\in M_n$ is a projection such that $\dist (P,Q)\geq \delta$, then $\rank (Q)\leq n - (1+\varepsilon')\rank(P)$.  The result will then follow from Definition \ref{def:neighborhood}.

Let $U_1,\dotsc,U_d \in M_n$ be unitaries such that $\Phi(X) = \frac{1}{d}\sum_{j=1}^d U_jXU_j^*$ for each $X \in M_n$.  If $\dist(P,Q) \ge \delta > 1$, it follows from the definition of the quantum metric induced by $\Phi$ that $PU_j^*U_i Q = 0$ for all $1 \le i , j\le d$.
Thus,
\begin{multline*}
\pair{\Phi(P)}{\Phi(Q)}_{\HS} = \frac{1}{d^2} \sum_{i,j=1}^d \pair{U_jPU_j^*}{U_iQU_i^*}_{\HS} \\
=  \frac{1}{d^2} \sum_{i,j=1}^d \tr\big( U_iQU_i^*U_jPU_j^* \big) = \frac{1}{d^2} \sum_{i,j=1}^d \tr\big( U_i^*U_jPU_j^*U_iQ \big) = 0.
\end{multline*}
 Therefore, by Lemma \ref{lemma-Cheeger} (Definition \ref{def:expander} implies that $\Phi$ is unital) and Remark \ref{remark:Cheeger},
 \begin{equation}\label{eqn-consequence-of-Cheeger}
 \varepsilon' \tr(P) \le \pair{\Phi(P)}{\Phi(I_n-P)}_{\HS} = \pair{\Phi(P)}{\Phi(I_n-P-Q)}_{\HS}.
\end{equation}
Now, by the trace duality between the trace and operator norms on $M_n$,
\begin{multline*}
\pair{\Phi(P)}{\Phi(R)}_{\HS} = \frac{1}{d^2} \sum_{i,j=1}^d \pair{U_jPU_j^*}{U_iRU_i^*}_{\HS}  =  \frac{1}{d^2} \sum_{i,j=1}^d \tr\big( U_iRU_i^*U_jPU_j^* \big)  \\
=  \frac{1}{d^2} \sum_{i,j=1}^d \tr\big( RU_i^*U_jPU_j^*U_i \big) \le  \frac{1}{d^2} \sum_{i,j=1}^d \n{R}_{1} \n{U_i^*U_jPU_j^*U_i }_\infty \le \n{R}_{1}
\end{multline*}
for all projections $R\in M_n$.
Therefore, using the fact that for projections the rank, the trace, and the trace norm coincide, it follows from \eqref{eqn-consequence-of-Cheeger} that
\begin{multline*}
\varepsilon' \rank(P)\le \n{I_n-P-Q}_{1} = \tr(I_n-P-Q)\\ = \tr(I_n) - \tr(P) - \tr(Q)=n - \rank (P)-\rank (Q),
\end{multline*}
which yields the desired inequality.
\end{proof}

We would like to use Proposition \ref{prop:expander_rank_inequality} to establish that quantum expanders satisfy a $(\delta, \varepsilon, t)$-isoperimetric inequality.  To do this, we have to first establish a relationship between the rank and the diameter of a projection inside an expander.  We do this more generally for projections inside any connected quantum graph and then show that expanders are connected.  The importance of the connectedness assumption is that it implies that every projection has finite diameter.  A quantum graph (that is, an operator system) $\mathcal{S}\subseteq M_n$ is said to be connected if there is $m\in \mathcal{N}$ such that $\mathcal{S}^m=M_n$ \cite[Definition 3.1]{alej2019connectivity}.  See \cite{alej2019connectivity} for more information about connected quantum graphs.

\begin{proposition}
\label{prop:connected_graph_diameter}
Let $\mathcal{S}$ be a connected quantum graph, and $R \in M_n$ a projection.  If $k\in \mathbb{N}$ is such that $\diam (R)\leq k$, then $R\mathcal{S}^kR=RM_nR$.
\end{proposition}

\begin{proof}
Suppose to the contrary that $R\mathcal{S}^kR \subsetneq RM_nR$, and pick any $A\in RM_nR\setminus R\mathcal{S}^kR$.  Then by \cite[Lemma 2.8]{Weaver2012Relations}, there exist projections $P,Q\in RM_nR\overline{\otimes} \mathcal{B}(\ell_2)$ such that $P(RAR\otimes \id)Q\neq 0$, while $P(RBR\otimes\id )Q=0$ for all $B\in \mathcal{S}^k$.  Let $\tilde{P}$, $\tilde{Q}$ be the range projections of $(R\otimes \id)P$ and $(R\otimes \id)Q$, respectively.  The above implies that $\tilde{P}(RAR\otimes \id)\tilde{Q}\neq 0$, while $\tilde{P} (B\otimes \id)\tilde{Q}=0$ for all $B\in \mathcal{S}^k$.  By Definition \ref{def:distance} and Definition \ref{def:diameter}, this means that $\diam (R)>k$.  This is a contradiction, and so $RS^kR=RM_nR$.
\end{proof}

\begin{lemma}
\label{lem:small_diameter_low_rank}
Let $\Phi \colon M_n \to M_n$ be CPTP map and let $K_1, K_2, \dots , K_N\in M_n$  be such that $\Phi(X)=\sum_{j=1}^N K_jXK_j^*$ for all matrices $X\in M_n$.  If the quantum graph associated to  $\Phi$ is connected, then for every projection $R \in M_n$,
\[
\rank(R) \le N^{\diam(R)},
\]
where the diameter is taken with respect to the quantum graph metric associated to $\Phi$.
\end{lemma}

\begin{proof}

Let $k = \diam(R)$
and let $\mathcal{S} = \spa\{ K_j^*K_i \mid 1\le i,j \le N \}$ be the associated quantum graph.
It follows from Proposition \ref{prop:connected_graph_diameter} that
$RM_nR = R\mathcal{S}^kR$
and therefore
\[
\rank(R)^2 = \dim( RM_nR ) = \dim( R\mathcal{S}^kR) \le \dim( \mathcal{S}^k) \le (N^2)^k,
\]
which yields the desired inequality.
\end{proof}

\begin{proposition}
Let $\Phi \colon M_n \to M_n$ be a CPTP unital map with an $\varepsilon$-spectral gap.
Then the associated quantum graph is connected.  In particular, every $d$-regular $\varepsilon$-quantum expander is connected.
\end{proposition}

\begin{proof}
Let $K_1,\dotsc,K_N \in M_n$ be matrices such that $\Phi(X) = \sum_{j=1}^N K_jXK_j^*$ for each $X \in M_n$.
Suppose that the quantum graph $\mathcal{S} = \spa\{ K_j^*K_i \mid 1\le i,j \le N \}$ is disconnected.
By \cite[Theorem 3.3]{alej2019connectivity}, there exists a nontrivial projection $P\in M_n$ such that $P\mathcal{S} (I_n-P) =0$,
and without loss of generality, we may assume $0<\rank(P)\le n/2$.
In particular, $PK_j^*K_i(I_n-P)=0$ for all $1\le i, j\le N$.
Therefore
\begin{multline*}
\pair{\Phi(P)}{\Phi(I_n-P)}_{\HS} = \sum_{i,j=1}^N \pair{K_jPK_j^*}{K_i(I_n-P)K_i^*}_{\HS} \\
=  \sum_{i,j=1}^N \tr\big( K_i(I_n-P)K_i^*K_jPK_j^* \big) = \sum_{i,j=1}^N \tr\big( K_i^*K_jPK_j^*K_i(I_n-P) \big) = 0.
\end{multline*}
This contradicts Lemma \ref{lemma-Cheeger}, and so the quantum graph associated to $\Phi$ is connected.
\end{proof}

Propositions \ref{prop:equi_coarse_infinite_dimension} and \ref{prop:expander_rank_inequality} and Lemma \ref{lem:small_diameter_low_rank} together yield our main theorem.

\begin{theorem}\label{thm:infinite-asymptotic-dimension}
If a quantum metric space $\mathcal{M}$ is equi-coarsely embeddable into a sequence of reflexive $d$-regular $\varepsilon$-quantum expanders, then $\asdim(\mathcal{M})=\infty$.  In particular, this holds whenever $\mathcal{M}$ admits a sequence of reflexive $d$-regular $\varepsilon$-quantum expanders as metric quotients.
\end{theorem}

We point out that, in particular, Theorem \ref{thm:infinite-asymptotic-dimension} covers the case when $\mathcal{M}$ equi-coarsely contains a sequence of $d$-regular $\varepsilon$-classical expanders, thanks to the following proposition.

\begin{proposition}
\label{prop:finite_is_reflexive}
The canonical quantum metric associated to a classical metric is always reflexive.
\end{proposition}

\begin{proof}
Let $(X,d)$ be a classical metric.
Let $t \ge 0$.
By \cite[Prop. 2.5]{KuperbergWeaver}, the canonical quantum metric on $\ell_\infty(X)$ associated to $d$ is given by
\begin{equation}\label{eqn:quantum-metric-associated-to-classical}
\mathcal{V}_t = \big\{ A \in \mathcal{B}(\ell_2(X)) \mid d(x,y) > t \Rightarrow \pair{Ae_y}{e_x} = 0\big\}.
\end{equation}
Let us now show that $\mathcal{V}_t$ is reflexive.
To that end, let $B \in \mathcal{B}(\ell_2(X))$ be such that for any projections $P,Q \in \mathcal{B}(\ell_2(X))$ such that $P \mathcal{V}_tQ = \{0\}$, it follows that $PBQ=0$; we need to show that $B$ belongs to $\mathcal{V}_t$.
Recall that $V_{xy}$ denotes the mapping $g \mapsto \pair{g}{e_y}e_x$.
Let $x,y \in X$ satisfy $d(x,y)>t$.
Note that $V_{xx}$ and $V_{yy}$ are projections, and it follows from \eqref{eqn:quantum-metric-associated-to-classical} that $V_{xx}\mathcal{V}_tV_{yy} = \{0\}$.
Therefore,  $V_{xx}BV_{yy} = 0$. But this implies $\pair{Be_y}{e_x} = 0$, and thus $B \in \mathcal{V}_t$ by apppealing to \eqref{eqn:quantum-metric-associated-to-classical} again.
\end{proof}

One final remark is in order regarding Theorem \ref{thm:infinite-asymptotic-dimension} and Proposition \ref{prop:finite_is_reflexive}.  We have shown that equi-coarse containment of reflexive quantum expanders implies infinite asymptotic dimension and we have also shown that quantum expanders induced by classical expanders are reflexive.  While this is enough to provide a generalization of \cite[Sec. 2.3]{Dranishnikov-Sapir} to the realm of quantum metric spaces, what we have not shown is the existence of a nontrivial reflexive quantum expander.  That is, we do not actually know whether every reflexive quantum expander is induced by a classical expander.  It would be interesting to know the answer to this question, but it would be more interesting still to know whether the reflexivity assumption in Theorem \ref{thm:infinite-asymptotic-dimension} (or more generally Proposition \ref{prop:equi_coarse_infinite_dimension}) can be dropped.  As with the proof of Theorem \ref{thm:nondisjoint_union}, it was very important to be able to place an upper bound on the diameter of a neighborhood of a projection in terms of the diameter of the projection (Part (c) of Lemma \ref{lemma:r-saturated-union}).  This is what allowed us to repeatedly apply the isoperimetric inequality to derive the inequality found in Remark \ref{remark:isoperimetric-repeat}.

\bibliography{references}
\bibliographystyle{amsalpha}

\end{document}